\documentclass{article}

\usepackage{
amsmath,
amsthm,
amscd,
amssymb,
}
\usepackage{stmaryrd}
\usepackage{comment}
\usepackage{tikz}
\usepackage{tikz-cd}
\usetikzlibrary{positioning,arrows,calc}
\usepackage{xspace}

\usepackage{graphicx}

\usepackage{url}

\theoremstyle{plain}
\newtheorem{lemma}{Lemma}
\newtheorem{definition}{Definition}

\newtheorem{theorem}{Theorem}

\newtheorem{remark}{Remark}

\newtheoremstyle{derp}
{3pt}
{3pt}
{}
{}
{\upshape}
{:}
{.5em}
{}
\theoremstyle{derp}

\newcommand{\R}{\mathbb{R}}

\newcommand{\Z}{\mathbb{Z}}

\newcommand{\N}{\mathbb{N}}

\newcommand{\sphere}{\mathbb{S}}
\newcommand{\Conv}[1]{\mathrm{Conv}(#1)}

\newcommand{\llb}{\llbracket}
\newcommand{\rrb}{\rrbracket}

\newcommand{\supp}{\mathrm{supp}}

\newcommand{\est}{\tikz[baseline=-3]{\node[draw,circle,inner sep = 1pt] () at (0,0) {\small$\rightarrow$};}}
\newcommand{\nth}{\tikz[baseline=-3]{\node[draw,circle,inner sep = 1pt] () at (0,0) {\small$\uparrow$};}}
\newcommand{\wst}{\tikz[baseline=-3]{\node[draw,circle,inner sep = 1pt] () at (0,0) {\small$\leftarrow$};}}
\newcommand{\sth}{\tikz[baseline=-3]{\node[draw,circle,inner sep = 1pt] () at (0,0) {\small$\downarrow$};}}

\title{Examples of Non-Busemann Horoballs}

\author{
Ville Salo \\
vosalo@utu.fi
}

\begin{document}
\maketitle

\begin{abstract}
In the f.g.\ setting, we construct horoballs which are ``not centered around a geodesic'' for generalized Heisenberg groups and all wreath products with an infinite acting group. That is, we find are limits of balls -- called \emph{horoballs} -- which are not increasing unions of balls around points on a geodesic -- called \emph{Busemann horoballs}. In the case of wreath products this follows from the stronger fact that there exist disconnected horoballs; on the lamplighter group we exhibit limits of Busemann horoballs which are not even coarsely connected. In the case of generalized Heisenberg groups, we use a symmetry argument instead; in fact for the $3$-dimensional Heisenberg group, at least under a particular generating set, we show that all horoballs are connected.

This note was written somewhat in isolation from the literature, in response to a recent preprint of Epperlein and Meyerovitch asking for examples of non-Busemann horoballs. It turns out that most of the results mentioned above are known: The results about connectedness follow from known results about almost convexity. The non-coarsely connected horoballs in the lamplighter group, and connected non-Busemann horoballs in the Heisenberg group, may be new statements, but here also we point out some strongly related literature.
\end{abstract}

\section{Introduction}

Let $G$ be a finitely-generated (f.g.) group and fix a generating set $S$ on $G$; we always assume $S$ is finite, \emph{symmetric} ($S = \{s^{-1} \;|\; s \in S\}$), and contains the identity $e = e_G$. An \emph{infinite geodesic} is $p : \N \to G$ such that $\forall n \in \N: d_S(p(0), p(n)) = n$, where $d_S$ denotes the word metric, and $B_{S,n}(g)$ is the $n$-ball around $g$ with respect to the word metric. On subsets of $G$ we use the Cantor topology.

It is natural to consider the horoballs of the group (seen as a metric space under the word metric), i.e.\ ``balls centered around a point at infinity''. There are many ways to formalize this. Our emphasis is on the following two.

\begin{definition}
Let $P \subset G$ with $P \neq \emptyset, G$. We say $P$ is
\begin{itemize}
\item a \emph{Busemann horoball} if there is an infinite geodesic $p$ such that $P = \bigcup_n B_{S,n}(p(n))$, and
\item a \emph{horoball} if $P = \lim_i B_{S,n_i}(g_i)$ for some $g_i \in G, n_i \in \N$ with $n_i \rightarrow \infty$.
\end{itemize}
\end{definition}

The notion of horoball is due to Gromov \cite{Gr81}, see \cite{EpMe20}. Our main motivation for studying horoballs is that the sets $\emptyset$ and $G$ are also limits of suitably positioned balls of increasing radii, so the horoballs and the sets $\emptyset$ and $G$ taken together form a $G$-subshift under translations, we call this the \emph{horoball subshift}.

These objects arise in some symbolic dynamical applications: In \cite{MeSa19} the author and Meyerovitch used them to study subshifts consisting of periodic points. In the recent preprint \cite{EpMe20}, Epperlein and Meyerovitch study the cellular automaton $f : \{0,1\}^G \to \{0,1\}^G$ defined by $f(x)_g = \max_{s \in S}(x_{gs})$, observing that certain properties of $G$ (such as the growth rate and amenability) can be deduced from the dynamical properties of $f$. The connection to horoballs is that the limit set of this cellular automaton consists of precisely the unions of horoballs, equivalently unions of Busemann horoballs.

It was asked in \cite{EpMe20} whether on discrete groups all horoballs are Busemann. We wrote this paper somewhat in isolation in response to this question, and while it seems some of our results are not directly in the literature, most things are at least a near miss to being known -- in particular, our additional discussion of ball distortion in the Heisenberg group, the most technically challenging part, can be directly deduced in the literature. As our main goal was to resolve the question of \cite{EpMe20}, we have decided to simply point out these connections rather than rewrite a shorter (possibly empty) paper with only the new observations.

First, the question of \cite{EpMe20} can be answered by appealing to the literature, as follows: A necessary condition on a horoball being Busemann is that is is connected (since Busemann horoballs are increasing unions of balls), and by compactness (Lemma~\ref{lem:DisconnectedHoroball}) horoballs being connected is equivalent to distortion of intrinsic metrics of balls being bounded. This property has been introduced previously by Cannon \cite{Ca87} under the name \emph{almost convexity}. It is known that Thompson's $F$ \cite{ClTa03}, Baumslag-Solitar groups $BS(1,m)$ \cite{MiSh97} and fundamental groups of manifolds with the Sol geometry \cite{ShSt95} are not almost convex. It follows that these groups admit disconnected horoballs, thus non-Busemann horoballs.

Say a subset of a group $G$ is \emph{$S$-connected} if it is a connected subset of the Cayley graph with generating set~$S$.

\begin{theorem}
\label{thm:Wreath}
Let $H$ be a nontrivial f.g.\ group, and $K$ any infinite f.g.\ group. Then for any symmetric generating sets $S, T \ni e_G$, the group $G = H \wr K$ admits a $T$-disconnected $S$-horoball.
\end{theorem}

As observed above, this implies that such wreath products admit non-Busemann horoballs.
This result is not new. By \cite{Ba61}, no wreath products of the above type are finitely-presented, and according to \cite{ClTa05} it follows from \cite{Ca87} that every almost convex group is finitely presented groups. Combining these facts with the observation above (Lemma~\ref{lem:DisconnectedHoroball}), one obtains Theorem~\ref{thm:Wreath}. We give a direct proof.

We show a slightly refined version of the above theorem for the lamplighter group. The \emph{lamplighter group} $\Z_2 \wr \Z$ is a particular case of the previous theorem. It can be thought of as $X_0 \times \Z$ where $X_0 = \{x \in \{0,1\}^\Z \;|\; \sum x \leq \infty\}$, with the multiplication rule is $(x,m)(y,n) = (x+\sigma^m(y),m+n)$, $\sigma^m(y)_i = y_{i-m}$. It admits a natural homomorphism $\pi(x,n) = n : \Z_2 \wr \Z \to \Z$. Consider the generating set $a = (0^\Z, 1), b = (x, 1) $ where $x$ is the characteristic function of $0$.

Say metric space $(M,d)$ is \emph{coarsely connected} if for some $t > 0$, for all $x,y \in M$ there exist $x_0 = x, x_1, ..., x_k = y$ such that $d(x_i, x_{i+1}) < t$ for all applicable $i$. We say a subset $A$ of a group $G$ is coarsely connected, if it is coarsely connected under the induced metric from $G$ for some finite generating set $S$. This is independent of $S$, and in fact equivalent to being $S$-connected for some finite generating set $S$.

\begin{theorem}
\label{thm:Lamplighter}
With the generating set $\{e, a, b, a^{-1}, b^{-1}\}$ for $\Z_2 \wr \Z$, there is a horoball which is not coarsely connected, and is a limit of Busemann horoballs.
\end{theorem}

We note that it is certainly possible that a horoball is coarsely connected yet not connected, for example this happens when the proof of Theorem~\ref{thm:Wreath} is applied in the most obvious way to the lamplighter group. Thus, this theorem may be new. However, we note that \cite{ClTa05} performs a rather similar explicit study of the basic lamplighter group (under the other natural generating set), giving a formula for geodesics and explicit examples of dead ends.

An important symbolic dynamical consequence of Theorem~\ref{thm:Lamplighter} is that it implies that Busemann horoballs together with $\emptyset$ and $G$ need not form a subshift. We note the subtlety that a limit of Busemann horoballs can be assumed to come from a converging sequence of geodesics (by Lemma~\ref{lem:Grazing} and a compactnessa argument), but the map $\eta$ from geodesics to corresponding Busemann horoballs is not continuous. It is only lower semicontinuous (in the sense of \cite{Ku32}): if we have pointwise convergence $p_i \rightarrow p$ for geodesics $p_i, p$ (say, all starting from the origin) and $\eta(p_i) \rightarrow P$, then $P \supset \eta(p)$, thus limits of Busemann horoballs can be smaller than one would expect, and this is exactly what happens in the case of the previous theorem.

We also give an example of another type. The \emph{generalized (discrete) Heisenberg group} $H = H_{2n+1} \subset \Z^n \times \Z^n \times \Z$ containing elements $(a,b,c)$ with $2|c \iff 2|a \cdot b$. with product $(a,b,c) (a',b',c') = (a+a', b+b', c+c' + (ab' - a'b))$. This is a slightly nonstandard representation, obtained from the one in exponential coordinates, by multiplying the $c$-values by $2$ so they are integers. The case $H_3$ is usually called the discrete Heisenberg group. Let us say a generating set $S$ for $H_{2n+1}$ is \emph{projection-symmetric} if it is fixed under the \emph{projection flip} $(a,b,c) \mapsto (-a,-b,c)$.

\begin{theorem}
\label{thm:Heisenberg}
Let $S$ be any symmetric projection-symmetric generating set for the generalized Heisenberg group. Then there is an $S$-horoball $H$ which is not a limit point of Busemann horoballs.
\end{theorem}

This result may be new, since it cannot be deduced from connectedness issues (see below), and we are not aware of other previously studied concepts that would directly imply such a result. However, we note that the \emph{tools} we develop for its proof are likely not new. The essential idea is a (partial) classification of geodesics, and although we did not find our exact statements in the literature, the large scale behavior of geodesics in these groups seem to be well-understood, see \cite{DuMo14} or \cite{Pa83}.

Not surprisingly, the non-Busemann horoballs are obtained by taking a (subsequence of a) non-geodesic path, and balls around its elements. The obvious path to use is the one moving in the central direction, and we give an explicit computation of a non-Busemann horoball that can be obtained like this: $H_3 \setminus (\{0\}^2 \times \N)$ is a non-Busemann horoball under the generating set
\[ S = \{(0,0,0), (1,0,0), (0,1,0), (-1,0,0), (0,1,0)\}. \]
Note that this set is connected, giving examples of non-Busemann horoballs which are connected.

In fact, in the case of the three-dimensional Heisenberg group, under this generating set $S$, all horoballs are connected:

\begin{theorem}
\label{thm:HoroballsConnected}
For the Heisenberg group $H_3$ under the generating set $S$ from above, every horoball is $S$-connected.
\end{theorem}

This result is not new: As discussed above, this is equivalent to almost convexity, and almost convexity has been established in \cite{Sh89} for $H_3$ under the same generating set (and for some related groups as well). It has been shown that \cite{Th92} that for generalized Heisenberg groups, almost convexity can depend on the generating set.

In our proof, we use an explicit natural ``optimization formula'' for the word metric w.r.t.\ $S$ (which can be turned into a direct formula), to ensure that geodesics look like one would expect. A formula has been given previously in \cite{Bl03}. As discussed above, the shapes of geodesics have been studied in the literature in much more general contexts \cite{DuMo14,Pa83}.

\section{Definitions}

On $\mathcal{P}(G) = 2^G$ the topology is Cantor with the subbasis of cylinders $[g] = \{H \subset G \;|\; g \in H\}$. For a generating set $S$ on $G$, we always assume $S$ is \emph{symmetric}, $S = \{s^{-1} \;|\; s \in S\}$, and contains the identity $e = e_G$. A group $G$ becomes a (topologically discrete) metric space under the left-translation invariant \emph{word metric} $d(g,h) = d_S(g, h) = |g^{-1}h|$ where $|g| = |g|_S = \min \{\ell \;|\; \exists (s_i)_i \in S: g = s_1 s_2 \cdots s_\ell\}$, with balls $B_{S,r}(g) = \{h \in G \;|\; d(g, h) \leq r\}$.

Write $\llb a, b \rrb = [a, b] \cap \Z$, and similar notation is used for other types of intervals.
A \emph{path in $H \subset G$ (of length $k$ from $g \in H$ to $g' \in H$)} is $(g_0, g_1, g_2, ..., g_k)$ with $g_i \in H$ and $g_{i+1} \in g_i S$ for applicable $i$, $g_0 = g, g_k = g'$, and a \emph{path} is just a path in $G$. We also write paths as functions, $p : \llb 0, n \rrb \to G$ or $p : \N \to G$. We say a set $H \subset G$ is ($S$-)\emph{connected} if there is a path in $H$ between any two elements of $H$. A path is \emph{geodesic} if its length is equal to the distance between its endpoints. An \emph{infinite geodesic} is $p : \N \to G$ such that $p|_{\llb 0,n \rrb}$ is a geodesic for all $n \in \N$. The projection of a path with a group homomorphism is obtained by projecting the elements on the path.

Paths beginning at the origin of the group can be obtained by taking partial products of a sequence over a generating set: if $x \in S^n$ (resp. $x \in S^\omega$), we write $\hat x$ for the path $p(i) = x(0) x(1) x(2) \cdots x(i-1)$, with the interpretation $p(0) = e_G$. Conversely, we sometimes want to see paths as (possibly infinite) words over the generators, and we simply speak of paths \emph{written in terms of generators}. Write $|p|$ for the length of a finite path, i.e.\ $|p| = n$ where $p : \llb 0, n\rrb \to G$. 

Let $M_1, M_2$ be metric spaces under metrics $d_1, d_2$, with $M_1 \subset M_2$ (but $d_1$ not necessarily the restriction of $d_2$). Write
\[ \Delta^{d_2}_{d_1}(\ell) = \sup \{d_1(a, b) \;|\; a, b \in M_1, d_2(a, b) \leq \ell \} \in \R_+ \cup \{\infty\} \]
for the \emph{distortion function} of $M_1$ in $M_2$. We assume $d_1$ measures an intrinsic distance in $M_1$, and always have $d_1 \geq d_2$. In most of our applications, $M_2$ is a group and $M_1$ is a ball or a horoball of $M_2$ with respect to some generating set. Since balls are connected, the distortion functions of balls take only finite values.

We use the standard $O(\cdot)$ and $\Omega(\cdot)$-notations and conventions. We specify in words what is thought of as fixed, for the purpose of invisible constants. We write $a \pm O(1)$ when $a + c$ where $c \in \Z$ and $c = O(1)$. 

We assume some basic knowledge of convex geometry, but give the main definitions. The \emph{convex hull} of $A \subset \R^d$ is the smallest convex set containing $A$. The convex hull of a finite set $A$ is same as the set of convex combinations of points in $A$. For $A \subset \R^d$, write $A^k$ for the iterated Minkowski sum $A^k = A + A^{k-1}$, observe that if $A$ is convex then $A^k = kA$. The following is a special case of the Shapley-Folkman lemma:

\begin{lemma}
\label{lem:SF}
For any finite set $A \subset \R^d$ and $k \geq d$, we have $k\Conv{A} \subset A^{k-d} + d\Conv{A}$.
\end{lemma}

An \emph{affine half-space} is $\{v \in \R^d \;|\; v \cdot u \geq r\}$ for some $r \in \R$ and $u \in \R^d \setminus \{0^d\}$. The \emph{stretch} of an affine half-space is the minimal $|r|$ such that this holds for some unit vector $u$.

The \emph{first quadrant} of $\Z^2$ is $\N^2$, and the second, third and fourth quadrants are obtained by counterclockwise rotation by $\pi/2$, $\pi$ and $3\pi/2$, respectively. (Note that the quadrants intersect.)

\section{Horoballs for wreath products}

\subsection{General wreath products}

We need a simple compactness observation. In the following lemma, we consider balls with respect to a metric $S$, while distances are measured with respect to a possibly differen generating set $T$.

\begin{lemma}
\label{lem:DisconnectedHoroball}
Let $G$ be a f.g.\ group with symmetric generating sets $S, T \ni e_G$ and $d = d_T$ the word metric from $T$. For each $n$ let $d_n$ be the path metric for the ball $B_{S,n}(e_G)$, as an induced subgraph of the $T$-Cayley graph of $G$.
The following are equivalent:
\begin{itemize}
\item for all $\ell$, $\{\Delta^{d}_{d_n}(\ell) \;|\; n \in \N\}$ is infinite,
\item there is a $T$-disconnected $S$-horoball.
\end{itemize}
\end{lemma}

\begin{proof}
If $\Delta^{d}_{d_n}(\ell)$ is unbounded for fixed $\ell$, then for all $m$ we find $n_m$ and $a_m, b_m \in G$ such that $d(a, b) \leq \ell$ and $d_{n_m}(a_m, b_m) \geq m$. Note that then $B_{S,n_m}(e_G)$ cannot contain the $\ell$-ball around $a_m$, as soon as $m > \ell$. Then $B_{S,n_m}(a^{-1}_m)$ contains $e_G$ and $a^{-1}_m b_m$ (which has $T$-word norm at most $\ell$), but does not contain the $\ell$-ball around $e_G$. Thus any limit point as $m \rightarrow \infty$ is an $S$-horoball $P$. Let $a_m^{-1} b_m \rightarrow c$ along the corresponding subsequence. If $e_G$ were $T$-connected to $c$ in $P$ with distance $t$, then by the definition of the topology on $2^G$ these paths would be contained in $B_{S,n_m}(a^{-1}_m)$ for large enough $m$, and translating them by $a_m$ would give $d_{n_m}(a_m, b_m) \leq t$ for large enough $m$, a contradiction.

Conversely, if $\Delta^{d}_{d_n}(\ell) = O(1)$ for all $\ell$, and $P = \lim_i B_{S,n_i}(g_i)$ for some $g_i \in G, n_i \in \N$ with $n_i \rightarrow \infty$, consider any $a, b \in P$. Then $a, b \in B_{S,n_i}(g_i)$ for large enough $i$, and
\[ d_i(g_i^{-1} a, g_i^{-1} b) \leq \Delta^d_{d_i}(d(g_i^{-1} a, g_i^{-1} b)) = \Delta^d_{d_i}(d(a, b)) = O(1), \]
so there is a $T$-path of length $O(1)$ from $a$ to $b$ inside $B_{S,n_i}(g_i)$. There are finitely many possible paths, so in some subsequence the same path $p$ is used in all $B_{S,n_i}(g_i)$. Then $p$ is contained in $P$.
\end{proof}

We summarize the above lemma by saying that horoballs are connected if and only if balls have bounded distortion (in the sense that the distortion functions are pointwise bounded). 

\begin{theorem}
Let $H$ be a nontrivial f.g.\ group, and $K$ any infinite f.g.\ group. Then for any symmetric generating sets $S, T \ni e_G$, the group $G = H \wr K$ admits a $T$-disconnected $S$-horoball.
\end{theorem}

\begin{proof}
Let $t$ be such that $T \subset S^t$. Consider $K$ with the word metric $d_K$ arising from the symmetric generating set $S_K = \pi_K(S)$. Note that $G$ acts freely from the right on $H^K \times K$: $k \in K$ acts by $(x, k') \cdot k = (x, k'k)$, $h \in H$ acts by $(x, k') \cdot h = (y, k')$ where $y_{k'} = x_{k'} h$ and $\forall k \neq k': y_k = x_k$. We may identify $G$ with the pairs $(x, k)$ where $x_{k'} = e_H$ for $k'$ in a cofinite set, as $G$ clearly acts simply transitively on this set. We refer to the second component of $(x, k)$ as the \emph{head}, and $\supp(x,k) = \{k' \;|\; x_{k'} \neq e_H\}$ as the \emph{support}. For $h \in H$, define $\hat h = (x, e_K)$ where $x_{e_K} = h$, $x_k = e_H$ for $k \neq e_K$.

Pick some $h \in H$ once and for all, so $\hat h$ has word norm $O(1)$ in $G$. Pick a large $N$ and pick any $k_1, k_2 \in K$ such that $d_K(k_1, k_2) = 2N$ and $d_K(k_i, e_K) = N$. Pick any preimages $\pi(g_i) = k_i$ for $i = 1, 2$, with $d_S(e_G, g_i) = N$. This is possible since we chose $S_K = \pi(S)$ as the generating set of $K$. Observe that $\hat h^{g_1}$ and $\hat h^{g_2}$ commute if $N$ is large enough, as the support of $\hat h^{g_i}$ contains only elements at a bounded distance from $k_i$.

It is easy to see that $g = \hat h^{g_1} \circ \hat h^{g_2} \in G$ has word norm $4N \pm O(1)$: The upper bound comes from simply writing out the expression, and interpreting it as a path. The lower bound comes from the fact that even through the projection, the head has to travel to within a constant distance from $k_1$ and $k_2$, in some order.

Now let $c = \pm O(1)$ such that $|g| = 4N + c$. As in the previous paragraph, we see that picking $c' = O(1)$ sufficiently large, the ball of radius $4N + c - c'$ around $g$ contains no elements $(x, k)$ such that $k$ is at distance more than $N-t$ from both $k_1$ and $k_2$, and $x_{k_1} = x_{k_2} = e_H$ for both $i = 1, 2$.

On the other hand, by erasing the contents at $k_1$ and $k_2$ in two different orders, we find two elements $g'_1, g'_2 \in G$ in the $(4N + c - c')$-ball around $g$ where the heads are at bounded $S_K$-distance from each other, the head of $g'_i$ is at distance $N - t - O(1)$ from $k_i$, and the supports of the $g'_i$ are contained in an $S_K$-ball of radius $O(1)$. Clearly, their distance in $G$ is then bounded by a constant. By the previous paragraph, there is no $T$-path in $G$ between $g'_1$ and $g'_2$ where the values $x^1_{k_1} = x^2_{k_2} = e_H$ are not flipped back to $h$, since $d_K(k_1, k_2) = 2N$ and $T \subset S^k$ imply that an element of $T$ cannot move the head from the $(N-t)$-ball around $k_1$ into the $(N-t)$-ball around $k_2$. Thus, in any path within the ball the head must travel near either $k_1$ or $k_2$. This gives a lower bound on lengths of $T$-paths, which tends to infinity with $N$.

We obtain that the distortion function with respect to $T$ in such balls is unbounded in some fixed ball of radius $O(1)$. Any limit point is $T$-disconnected by the previous lemma. It is automatically a horoball because it does not contain $e_G$ and contains elements from a ball of radius $O(1)$.
\end{proof}

\subsection{The lamplighter group}

\begin{theorem}
With the generating set $\{e, a, b, a^{-1}, b^{-1}\}$ for $\Z_2 \wr \Z$, the set
\[ P = \{ g \in \Z_2 \wr \Z \;|\; \pi(g) < 0 \} \]
is not coarsely connected, and is a limit of Busemann horoballs.
\end{theorem}

\begin{proof}
For the following discussion, think of the lamplighter group as the group acting faithfully on $X_0 \times \Z$ where $X_0 = \{x \in \{0,1\}^{\Z+\frac12} \;|\; \sum x < \infty \}$, generated by $a(x,n) = (x,n+1)$ and $b(x, n) = (y, n+1)$ where $y_{n+1/2} = 1-x_{n+1/2}$, $y_i = x_i$ for $i \neq n+1/2$. The support of $(x,n)$ is of course $\{i \in \Z + \frac12 \;|\; x_i \neq 0\}$. We imagine $\Z+1/2$ as a horizontal number line ($-\infty+1/2$ is on the left), on which the head walks, flipping lamps on its way.

The fact $P$ is not coarsely connected is obvious: without moving the head to the right of $-1 \in \Z$, we cannot modify bits of the support at $n+1/2$ for large $n$.

The $\N$-parametrized geodesics of a group, from the origin, written over a generating set, form a subshift, which we call the \emph{geodesic ($\N$-)subshift}. It is easy to see that the geodesic subshift of the lamplighter group with respect to the above generators is the sofic shift which is the closure of the $\omega$-regular language
\[ (a + b)^* (ab^{-1} + ba^{-1}) (a^{-1} + b^{-1})^\omega + (a^{-1} + b^{-1})^* (b^{-1}a + a^{-1}b) (a + b)^\omega. \]
In other words, the head makes a finite (possibly empty) sweep on the left, flips a lamp, and then keeps traveling to the right, or vice versa.

Consider an origin-grazing ball around an element $g$, where $\pi(g) = N > 0$ is large, and the support is contained in $\llb -N+1/2, N-1/2 \rrb$ with leftmost element $-N+1/2$. Clearly the optimal way to reach the origin of the lamplighter group is to move over the entire support, and come back to the origin, so the word norm of $g$ is exactly $3N$. Clearly in every element of the $(3N-1)$-ball, either the head is strictly to the left of the origin or the $(-N+1/2)$th bit of the support is $1$. On the other hand, every lamplighter group element $(x, n)$ with $-N \leq n < 0$ and $\supp(x, n) \subset [-N, N]$ is in the ball.

It follows that as $N \rightarrow \infty$, these balls give the horoball $P$: If $n < 0$, then $(x,n)$ is in the ball once $N \geq -n$ and $\supp(x,n) \subset [-N, N]$, while if $n \geq 0$ then $(x,n)$ is not in the ball once $N+1/2 < \min \supp(x, n)$, since every element $(y, n)$ in the ball has $y_{N+1/2} = 1$.

By the description of the geodesic subshift of the lamplighter group, the inverses of the geodesics we described from the elements $g$ to the origin can be continued to infinite geodesics, and it is easy to see that this extension does not change the corresponding Busemann horoball's intersection with the $N$-ball. Thus, the limit of these Busemann horoballs is the same as the above limit of balls, namely $P$.

On the other hand, if $P$ were a Busemann horoball, up to symmetry we may assume that it is given by a geodesic in the closure of $(a^{-1} + b^{-1})^n a^{-1}b (a + b)^\omega$. Clearly no element $(x, -1)$ with $x_{n-1/2} = 1$ is in the corresponding Busemann horoball, thus it cannot be $P$.
\end{proof}

If one picks a converging subsequence of the geodesics to the elements $g$, one obtains a geodesic in $\{a^{-1}, b^{-1}\}^\N$. The corresponding Busemann horoball is contained in
\[ \{ h \in \Z_2 \wr \Z \;|\; \pi(h) < 0, \supp(h) \subset (-\infty, -1/2 \rrb \}, \]
thus is not equal to $P$.

\section{Horoballs for the generalized Heisenberg groups}

\subsection{Projection-symmetric horoballs}

We show that there exist horoballs for the generalized Heisenberg group which are symmetric under the projection flip, and that no such set can be a Busemann horoball. We need a lemma for discrete geodesics in $\Z^d$: after finitely many steps, any geodesic travels at maximal speed.

\begin{lemma}
\label{lem:ZdGeodesics}
Let $0^d \in S \subset \Z^d$ be any finite symmetric generating set. Then there exists $k$ such that for every S-geodesic $p : \N \to \Z^d$ with $p(0) = 0^d$, there exists a nonzero linear map $\ell : \R^d \to \R$ such that
\[ \forall t: \ell(p(t)) \geq \max_{s \in S}(t\ell(s)) - k. \]
\end{lemma}

\begin{proof}
We first prove that $k$ exists for each $p$ separately. Let $p = \hat x$ and consider the set of vectors $V \subset S$ used infinitely many times in $x$.

Observe that the convex hull of $V$ does not contain $0^d$: otherwise by linear algebra there is a positive $\Z$-linear combination of vectors in $V$ summing to $0^d$, so $x$ is not a geodesic because we can remove any corresponding subsequence.

We now claim that the convex hull of $V$ cannot contain any vector $v$ which is strictly contained in the convex hull of $S$. Otherwise, it also contains a rational convex combination $v = \sum_{j = 1}^d a_j v_j$. Let $k$ be such that $ka_j \in \Z$ so that $k v \in \Z^d$. Since $\sum_{j = 1}^d a_j = 1$, $\sum_{j = 1}^d ka_j = k$ and we have $kv \in V^k$.

By the assumption that $v$ is strictly contained in $\Conv{S}$, for some rational $\lambda > 1$ we have $\lambda v \in \Conv{S}$, thus $\lambda i k v \in \Conv{S}^{ik}$ for all $i \geq 1$, where $A^n$ for $A \subset \R^d$ denotes the iterated Minkowski sum, so by Lemma~\ref{lem:SF} we have
\[ \lambda ikv \in S^{ik-d} + u \]
for some $u \in \Conv{S}^d$.

We have $ikv \in V^{ik}$ for all $i$. If $m$ and $i = m/\lambda$ are integers then $\lambda ikv \in \Z^d$ and by the previous paragraph $\lambda ikv \in S^{ik-d} + u$, necessarily for some $u \in \Z^d$, with word norm $O(1)$ over the generating set $S$. Thus $\lambda ikv = km v \in V^{mk}$ has word norm at most $ik - d + O(1) = km / \lambda + O(1)$ over $S$. For large $m$, this clearly contradicts the assumption that $x$ is geodesic, as we can replace a subsequence corresponding to $kmv \in V^{mk}$ in $x$ by an element of $S^{km / \lambda + O(1)}$.

The fact $k$ is uniform is obtained from the same argument, considering instead the set of vectors that appear sufficiently many times so that the above shortcutting arguments apply.
\end{proof}

The \emph{natural projection} is $\phi : H_{2n+1} \to \Z^{2n}$ defined by $\phi(a,b,c) = (a,b)$. The \emph{height} of $(a,b,c)$ is $c$.

\begin{lemma}
Fix any finite generating set $S$ for $H_{2n+1}$. Let $x \in S^\omega$ be a geodesic starting at the origin written in terms of generators. Then $x = w \cdot y$ where the natural projection of $\hat y$ is a geodesic in $\Z^{2n}$ under the projections of the generators, and $|w|$ is bounded.
\end{lemma}

\begin{proof}
First let us prove this without a uniform bound on $|w|$. Suppose $x$ is a geodesic. If the natural projection of $\hat x$ does not travel arbitrarily far, then the height also grows at a linear rate. Then it is easy to see that $\hat x$ is not a geodesic, because we can raise the height by a quadratic amount in linearly many steps.

Suppose then that the projection of $\hat x$ does travel arbitrarily far, and the conclusion of the lemma is false. Looking only at the natural projection $\pi(\hat x)$, the assumption is precisely that we can take arbitrarily good shortcuts in the natural projection. When we do so, the actual path in the generalized Heisenberg group may have incorrect height after the shortcut. We show that that we can fix it any height problems by an $O(1)$-modification to the path, contradicting the assumption that $\hat x$ was a geodesic.

Since $\hat x$ travels arbitrarily far in the projection, along at least one coordinate among the $2n$ many, it travels arbitrarily far. It follows that there are subpaths whose offset corresponds to an element $(a,b,c)$ where at least one coordinate of $a$ or $b$ is unbounded (thus takes every value in a syndetic subset of $\N$ or $-\N$, the maximal possible gap size depending on $S$). The formulas
\[ (-v_i,0,0)(a,b,c)(v_i,0,0) 
= (a,b,c - 2 v_i \cdot b), \]
\[ (0,-v_i,0)(a,b,c)(0,v_i,0) 
= (a,b,c + 2 a \cdot v_i), \]
where $v_i$ is the $i$th standard generator of $\Z^n$ or its negation, now imply that we may raise or lower the projection by any amount from a syndetic subset of $\Z$, with only an $O(1)$ increase in word norm. The finite height offset that is left can be covered in $O(1)$ additional steps. This contradicts geodesicity as explained in the previous paragraph.

The fact there is a uniform bound on $|w|$ follows easily from a quantitative version of this argument.
\end{proof}

It should be possible to refine the above result into an exact description of the infinite geodesics, at least for the generating set used in the following sections. This subshift is not sofic, but we conjecture the finite prefixes form a one-counter language. The language of finite geodesics has been described rather completely in \cite{Sh89} for this generating set, and of course the infinite geodesics are just the limits of finite ones in the obvious sense.

We need a simple general lemma about Busemann horoballs. Say a set $P \subset G$ \emph{grazes} $g \in G$ if $g \notin P$ and $gS \cap P \neq \emptyset$. 

\begin{lemma}
\label{lem:Grazing}
If a Busemann horoball $P$ grazes $g$, then there is a geodesic that begins in the $S$-neighborhood of $g$ whose corresponding Busemann horoball is $P$.
\end{lemma}

\begin{proof}
Suppose $P = \bigcup_n B_{S,n}(p(n))$ for $p : \N \to G$ an infinite geodesic. Since $P$ grazes $g$, for some $n$ we have $B_{S,n}(p(n)) \cap gS \ni gs$ for some $s \in S$. Let $w : \llb 0, n \rrb \to G$ be a geodesic with $w(0) \in gs$ and $w(n) = p(n)$ and define a path by $q : \N \to G$ by $q(i) = \left\{\begin{array}{ll}
w(i) & \mbox{if } i \leq n \\
p(i) & \mbox{otherwise.}
\end{array}\right.$

Then $q$ must be a geodesic: Otherwise $d(q(0), q(j)) < j$. Since $q$ is a concatenation of two geodesics, we must have $j > n$ and thus $d(q(0), q(j)) = d(gs, p(j)) < j$. It follows that $d(g, p(j)) \leq j$ so $g \in B_{S,j}(p(j)) \subset P$, contradicting the assumption that $P$ grazes $g$.

The Busemann horoball corresponding to $q$ must be the same as that for $p$ since $q(i) = p(i)$ for large enough $i$, and the union defining a Busemann horoball is increasing.
\end{proof}

\begin{lemma}
For any symmetric projection-symmetric generating set $S \ni e_{H_{2n+1}}$ for $H_{2n+1}$, the natural projection of any limit of Busemann horoballs grazing the identity is contained in an affine half-space with bounded stretch.
\end{lemma}

\begin{proof}
If a Busemann horoball grazes the identity, the corresponding geodesic $p$ may be assumed to begin in the neighborhood of the identity by Lemma~\ref{lem:Grazing}. By the previous lemma, the natural projection $p$ of any geodesic starting at the origin, after a bounded prefix, becomes a geodesic in $\Z^{2n}$; clearly the same is true if $p$ begins in the neighborhood of the origin. Lemma~\ref{lem:ZdGeodesics} implies that
$\ell(p(t)) \geq \max_{s \in S}(t\ell(\pi(s))) - k$
for some nonzero linear map $\ell$ (we need to increase $k$ by $O(1)$ to account for the bounded shift in the starting point of the geodesic). Since $\pi(B_{S,r}(g)) = B_{\pi(S),r}(\pi(g))$ for all $r \in \N, g \in H_{2n+1}$, this formula clearly implies that the projection is contained in an affine half-space of bounded stretch, namely the half-space $\{v \in \R^d \;|\; \ell(v) \geq -k\}$. This property is clearly preserved under taking limits, since there is a uniform bound on $k$.
\end{proof}

We can now easily prove Theorem~\ref{thm:Heisenberg}.

\begin{theorem}
\label{thm:HeisenbergProof}
Let $S$ be any symmetric projection-symmetric generating set for the generalized Heisenberg group. Then there is an $S$-horoball $P$ which is not a limit of Busemann horoballs.
\end{theorem}

\begin{proof}[Proof of Theorem~\ref{thm:Heisenberg}]
Consider any set $P \subset H$ obtained as a limit point of balls of radius $|(0,0,c)|_S-1$ around $(0,0,c)$. Such $P$ contains some element of $S$ but not $(0,0,0)$, so it is a horoball. The projection of this horoball is not a limit of Busemann horoballs: by the lemma before the previous one, the projection of any geodesic from the $S$-neighborhood of $(0,0,0)$ must travel away from the origin at linear rate, so $P$ must contain $(a,b,c)$ with $|a|+|b|$ arbitrarily large. By projection symmetry of the generating set and the fact the projection flip fixes $(0,0,c)$, the projection flip fixes the balls around $(0,0,c)$, and thus $P$. It follows that $P$ is not a limit of Busemann horoballs by the previous lemma.
\end{proof}

\subsection{Explicit non-Busemann horoball in the Heisenberg}

We give a concrete example of what the horoball described in the previous section may look like, for the three-dimensional Heisenberg group $H = H_3$ under the generating set $S = \{(0,0,0), (1,0,0), (0,1,0), (-1,0,0), (0,-1,0)\} = \{e, \est, \nth, \wst, \sth\}$. We refer to the directions in $S$ informally as east, north, west and south; up and down refer to height, except that \emph{going down a geodesic} means walking along it. It is useful to think of elements of the Heisenberg group as paths moving in cardinal directions on the plane by unit-length moves, and identifying paths when they evaluate to the same element, i.e.\ reach the same element when lifted to the unique path from $e_{H_3}$.

By the \emph{double area} of a closed curve $p : \sphere^1 \to \R^2$ we mean twice the signed area of the region it delimits counterclockwise. It is well-known that the path $\hat w$, with $w \in S^*$ reaches precisely the element $(a,b,c)$, where $(a,b)$ is the endpoint of the path $\pi(\hat w)$ on the plane, and $c$ is the double area of the path in $\R^2$ obtained by composing $c$ with the straight line from $(a,b)$ to $(0,0)$. (Usually, area is used instead of the double area, but we prefer to work with integers rather than half-integers.)

It helps to have a formula for the word metric, and we give one.

\begin{lemma}
\label{lem:WordMetric}
Let $(a,b,c) \in H$. Then
\[ |(a,b,c)| = \min \{ 2(A + B) - (|a| + |b|) \;|\; 2AB - |a||b| \geq |c|, A \geq |a|, B \geq |b| \}. \]
\end{lemma}


\begin{proof}
By some symmetry considerations, it is enough to consider the case $(a,b,c)$ with $a,b,c \geq 0$. Denote the RHS of the formula by $f(a,b,c)$. For elements $(a,b,c)$ with $c \leq ab$ the claim is trivial: Then $f(a,b,c) = a+b$, and even in the natural projection, we cannot do better than this, so this is a lower bound. For the upper bound, it is easy to prove that any double area $c \leq ab$ can be reached by a geodesic that only moves north and to the east.

For $c \geq ab$, it is easy to see that for any $A \geq a, B \geq b$ there is a path of length $2(A + B) - (a + b)$ with double area $2AB - ab$, whose shape is a rectangle with a triangle cut off from the northwest corner. We can drop its area to $c$ without increasing its length by flipping counterclockwise turns to clockwise turns starting from the southeast corner.

To see the lower bound, we prove that movement by any of $\est,\nth,\wst,\sth$ can decrease the word norm by at most one. This can be checked by a straightforward but relatively long case analysis. These calculations can be found in Section~\ref{sec:Calculations}.
\end{proof}

\begin{remark}
While the formula is stated as an optimization problem, the solution is obvious -- $A$ and $B$ should be taken as close to each other as possible under the constraints, to maximize the area. We will argue in terms of the optimization problem, but one can write out a more explicit formula easily by working out some details parity. For $a \geq b \geq 0$ and $c \geq 0$, the following formula can be obtained:
\[ |(a,b,c)| = \left\{\begin{array}{ll}
a+b & \mbox{if } c \leq ab \\
2 \lceil \frac{c - ab}{2a} \rceil + a + b & \mbox{if } ab < c \leq 2a^2 - ab \\
2 (\lceil n/2 \rceil + \lfloor n/2 \rfloor) - a - b & \mbox{otherwise, where } n = \left\lceil 2 \sqrt{\frac{c+ab}{2}} \right\rceil
\end{array}\right. \]
The three regions correspond to the three ``regimes'' analyzed in the following section.
\end{remark}

A (presumably, and hopefully, equivalent) formula appears in \cite{Bl03}; we have not compared the proof to ours. The idea is fully explained in an answer of Derek Holt's on math.stackexchange.com \cite{Ho12}, and Holt suggests that ``It would be an interesting programming exercise and hard to get exactly right.'' We indeed performed such a programming exercise while writing this paper, to be able to stare at Heisenberg's balls; it is debatable whether it was interesting or hard.

\begin{theorem}
For $H = H_3$ under generating set $S = \{e,\est,\nth,\wst,\sth\}$, the set $P = H \setminus (\{0\} \times \{0\} \times \N)$ is a non-Busemann horoball.
\end{theorem}

\begin{proof}
This is not a Busemann horoball by the previous section. We show that it is a horoball. Define the \emph{columns} of a subset of $H$ as the intersections with sets of the form $\{x\} \times \{y\} \times \Z$. Observe that by the formula for the word metric, the column at $(x,y)$ of a ball centered at the origin is of the form $\{x\} \times \{y\} \times \llb -n, n \rrb$.

Consider now the function $\eta_n(x,y) = \max \{z \;|\; (x, y, z) \in B_{S, n}\}$. From the formula for the word metric, we see that for bounded $|x|, |y|$ and large $n$, $\eta_n(x,y)$ depends only on $|x| + |y|$ up to an additive constant $C$ (which does not not depend on $n$): if $x, y \geq 0$ then substituting $x \rightarrow x+1$, $y \rightarrow y-1$, as long as we stay in the same quadrant the cost $2(A+B) - (x+y)$ does not change for any $A, B$, and the reached height changes by $2AB - xy - (2AB - (x+1)(y-1)) = (x+1)(y-1) - xy$, which is bounded by a constant (optimal solutions satisfy $A > x, B > y$ anyway, since the $z$ reaching the height is much larger than $x + y$). Now walk around the $\ell^1$-sphere on the plane and apply this argument on the boundedly many step.

On the other hand, if $n$ is large and odd, and $|x| + |y|$ is small, then $\eta_n(x,y) \geq \eta_n(0,0) + \Omega(n) - O(1)$: In optimal solutions, $A, B$ are much larger than $x, y$ anyway, and because $2(A + B) - (|x| + |y|)$ decreases as $|x| + |y|$ increases and $n$ is odd, we can increase $A$ or $B$ by at least one, leading to a linear increase in $2AB - ab$ since $A$ and $B$ are $\Omega(n)$. 

The ball around $(0,0,-2N^2-1)$ grazing the origin has radius $4N+2$: if $2(|A| + |B|) \leq 4N+1$ then actually $|A| + |B| \leq 2N$ and we can reach at most the product $2AB = 2N^2$, while with radius $4N+2$ we can set $A = 2N+1, B = 2N$ to reach $2(N+1)N = 2N^2 + 2N > 2N^2+1$. Since $(0,0,-1)$ is in such a ball and columns are intervals, $(0,0,i)$ is in the ball for $i \in \in \llb -2N^2-1, -1 \rrb$, while no $(0,0,i)$ with $i \geq 0$ is in the ball for the same reason. By the previous paragraph, all elements with $(x, y, n)$, $|x|+|y| \neq 0$ are eventually in such balls, and we conclude that in the limit as $N \rightarrow \infty$ we obtain the horoball $P$. 
\end{proof}

\subsection{Heisenberg's horoballs are connected}

We now give a proof that the three-dimensional Heisenberg group, under the generating set $S$ from the previous section, has connected horoballs. As discussed, what needs to be shown is that distortion in balls is bounded. This result has been previously proved by Shapiro \cite{Sh89}.

\begin{theorem}
For $S = \{e, \est, \nth, \wst, \sth\}$, every $S$-horoball in the Heisenberg group $H_3$ is connected.
\end{theorem}

\begin{proof}
It is enough to show that the distortion functions $m \mapsto \ell(m)$ of balls are pointwise bounded, and for this is it enough to show that for every $m$, there exists $n_0$ such that for every $n \geq n_0$, and any $g, g' \in B_{S,n}$ with $d(g,g') \leq m$, there is a path of length at most $\ell$ inside $B_{S,n}$ from $g$ to $g'$. Fix an $m$ once and for all; we will prove that $\ell$ and $n_0$ exist. We allow $O(\cdot)$- and $\Omega(\cdot)$-notations to depend on $m$.

The crucial observation is the following. Suppose $g = (a,b,c), g' = (a',b',c')$ are in the $n$-ball, and $d(g, g') \leq m$. If we can find geodesics $p_g : g \to (0,0,0)$ and $p_{g'} : g' \to (0,0,0)$ such that $d(p_g(i), p_{g'}(i)) \leq r$ for $i \in \llb 0, k \rrb$, and $k \geq C \sqrt{rk}$ for an absolute constant $C$, then $g$ and $g'$ are connected in the $n$-ball. Namely, follow the geodesics back for $k$ steps, and observe that by the interpretation of the Heisenberg group as a central extension of $\Z^2$ by a double area calculating cocycle, we can make up for the height difference by walking around a discrete approximation of a square. This fits in the $n$-ball, since going down a geodesic by $k$ steps necessarily moves $k$-deep into the ball.

We will fix for each $g = (a,b,c)$ a \emph{canonical geodesic} $p_g$ from $g$ to the origin. In most cases the canonical geodesics automatically satisfy the above, and we deal with the remaining cases separately.

Instead of giving the family of canonical geodesics directly, we work with a real relaxation: we pick canonical ``pseudogeodesics'' for all $(a,b,c) \in H_3$, which are piecewise linear curves in $\R^2$. It is useful to consider the curves with the $\ell^1$-parametrization, so infinitesimal movement in direction $(a,b) \in \sphere^1$ has instant cost $|a|+|b|$. We describe the pseudogeodesic for $a, b, c \in \R_{\geq 0}$, as $c$ grows from $0$ to infinity.

If $c \leq ab$, then let $L \subset \R^2$ be the unique piecewise linear path having four linear pieces, in directions $(1,0)$, $(a,b)$, $(0,1)$, $(-a,-b)$, which has double area $c$. The pseudogeodesic is the initial three pieces (in directions $(1,0)$, $(a,b)$, $(0,1)$) of $L$. The pseudogeodesics of this type are said to be in the \emph{first regime}.

Formulas for the points where the pseudogeodesic turns can be obtained by solving a quadratic, they are $(d, 0)$ and $(a,(1-d/a)b)$ where $d = a - \sqrt{a(a - c/b)}$. For $(10,10,75)$ we have $c = 75 \geq 100 = ab$ so we are in the first regime, and the pseudogeodesic looks like so:
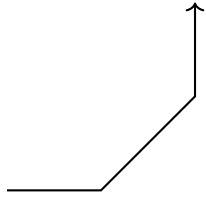
\begin{center}
\begin{tikzpicture}[scale=0.25]
\draw[thick,->] (0,0) -- (5,0) -- (10,5) -- (10,10);
\end{tikzpicture}
\end{center}

At $c = ab$, the path $L'$ is an actual geodesic, $\est^a \nth^b$ as a word. From this point on, as the pseudogeodesic we use will correspond to optimal solutions to the optimization problem $\min \{2(A+B)-(a+b) \;|\; 2AB - ab \geq c, A \geq a, B \geq b\}$. It is not hard to see that there is a unique solution, namely if $m = \max(a,b)$ and $c \leq 2m^2 - ab$, if $m = a$, we should set $A = a$ and pick $B$ so that $2aB = c + ab$, in other words $B = (c + ab)/2a$. The pseudogeodesic is the piecewise linear curve $(0,0) \rightarrow (0,b-B) \rightarrow (a,b-B) \rightarrow (a,b)$. Symmetrically if $m = b$ we should set $B = b$ and $A = (c + ab)/2b$, and the piecewise linear curve should be $(0,0) \rightarrow (a-A,0) \rightarrow (a-A,b) \rightarrow (a,b)$. These are pseudogeodesics in the \emph{second regime}.

For example for $(10,5,100)$ we have $ab = 50 \leq 100 \leq 150 = 2m^2 - ab$, so we are in the second regime. The pseudogeodesic is $(0, 0) \rightarrow (0, -2.5) \rightarrow (10, -2.5) \rightarrow (10, 5)$:
\begin{center}
\begin{tikzpicture}[scale=0.25]
\draw[thick,->] (0, 0) -- (0, -2.5) -- (10, -2.5) -- (10, 5);
\end{tikzpicture}
\end{center}

If $c \geq 2m^2 - ab$, we should always set $A = B = \sqrt{\frac{c + ab}{2}}$, and the pseudogeodesic is the corresponding square, with the northwest triangle cut out, i.e.\ $(0,0) \rightarrow (0,b-B) \rightarrow (a-A, b-B) \rightarrow (a-A,b) \rightarrow (a,b)$. These are pseudogeodesics in the \emph{third regime}.

For example for $(10,10,150)$ we have $100 = c \leq 2m^2 - ab = 150$, so we are in the third regime. The pseudogeodesic is roughly $(0, 0) \rightarrow (0, -1.18) \rightarrow (11.18, -1.18) \rightarrow (10, 10)$:
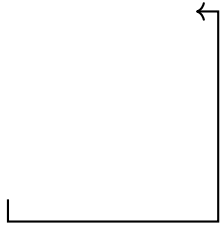
\begin{center}
\begin{tikzpicture}[scale=0.25]
\draw[thick,->] (0,0) -- (0, -1.18) -- (11.18, -1.18) -- (11.18,10) -- (10,10);
\end{tikzpicture}
\end{center}

The pseudogeodesic to $g$ can be turned into an actual geodesic $p_g$ from $(0,0,0)$ to $g$ when $g \in H_3$: When $c \geq ab$, replace the path by a ``digital approximation'' $D \in \{\nth, \est\}^{a+b}$ of $L'$, which at all times stays at bounded from $L'$ in the Euclidean metric, and which has the same double area. To do this, observe that one can find approximations with more and with less double area by having the path follow $L'$ strictly above or below, and by turning right turns $\nth\est$ into left turns $\est\nth$, one can get the correct double area. When $c \geq ab$, again use a rectangle whose side lengths are best approximations to the pseudogeodesic, and add a single long \emph{bump} on the longer side if needed, for instance changing $\est^n$ to $\sth \est^k \nth \est^{n-k}$ changes the length by $2$ and the area by $2k$. Such paths are a subclass of the geodesics we described in Lemma~\ref{lem:WordMetric}, thus they are geodesics. It is clear that the canonical geodesics then stay at a bounded distance from the pseudogeodesic in the Hausdorff metric, and because of the choice of $\ell^1$-parametrization for the pseudogeodesic, they stay close also as parametrized curves.

In this desciption we have assumed $a, b, c \geq 0$. For $c < 0$ the process is inverted, and we use piecewise linear curves above the diagonal, otherwise the process is the same. For other octants the pseudogeodesics are obtained by rotating the ones for the first.

In the argument that follows, we use some jargon. The word metric is defined by optimizing $A, B$ under some constraint. We refer to $A, B$ as the \emph{side lengths}. In the second and third regime, it is useful to think of the formula for the word metric of $(a,b,c)$ as being solved in a dynamical way, happening in discrete steps. First, we set the side lengths $A = a, B = b$. Then, as long as $c < 2AB - ab$, we keep incrementing one side by one, and we call each such increment a \emph{steps}. The important thing to note is that finding the optimum of the discrete optimization problem is as simple as always incrementing the shorter side, breaking ties arbitrarily. We refer to this as the \emph{optimization process}. If $(a,b)$ is in the first quadrant, we refer to $B - b$ as the \emph{excess} (independently of the sign of $a - b$). When $b \leq a$, excess is a sign of being in the third regime.

First suppose $|a|, |b|$ are large and suppose $|c| \geq t(|a| + |b|)$ and $|c'| \geq t(|a'| + |b'|)$ for a large $t$ (depending on $m$). We consider the case $a, b, c > 0$, as other cases are symmetric. Now consider the optimization process for the canonical geodesics for $(a,b,c)$ and $(a',b',c')$. Since $c' = c \pm O(a + b)$, $a' = a \pm O(1)$ and $b' = b \pm O(1)$, the optimization processes halt $O(1)$ steps apart: once the canonical geodesic for $(a,b,c)$ has been reached, the process for $(a',b',c')$ is at most $O(1)$ steps behind (if $a' < a$ or $b' < b$). In $O(1)$ steps the side lengths $A'$ and $B'$ used for $(a',b',c')$ will necessarily reach $a'$ and $b'$, and after that grows by $\Omega(a + b)$ on each step. In particular, the excess differs by $O(1)$.

Now, if $(a,b)$ and $(a',b')$ are in the same quadrant, we are done: Suppose both are in the first quadrant. Then both paths move to the east for $x$ steps, where $x$ is the excess (so the difference in these numbers of steps is $O(1)$); and then move to the south for a number of steps that grows to infinity as a function of $t$ and $|a| + |b|$ (in the second and third regime, this number of steps grows with $a + b$, and in the first regime with $t$). If $t$ and the lower bound on $a + b$ are large enough, we can apply the observation from the second paragraph in this case.

Now suppose that still $|c| \geq t(|a| + |b|)$ and $|c'| \geq t(|a'| + |b'|)$, and $|a|, |b|$ are large, but that $(a,b)$ and $(a',b')$ are on different quadrants. By symmetry we may suppose $(a',b')$ is on the first quadrant and $(a,b)$ on the fourth, so $a, a' > 0$. In this case, $|b|, |b'| = O(1)$. The previous argument works for the most part, but we have an additional problem: if one of $(a,b,c)$, $(a',b',c')$ is deep into the third regime, i.e.\ the excess is very large (so in fact both have to be in the third regime), then the canonical geodesic for $(a,b,c)$ initially moves down (possibly after a bump of thickness one), and the canonical geodesic for $(a',b',c')$ moves to the right.

In this case, first suppose that there is no bump in the canonical geodesic. Then move down in height from $(a,b,c)$ to $(a,b,c-2p)$ for sufficient $p$ so that the canonical geodesic can be shifted according to the following figure to a geodesic for $(a,b,c-2p)$:
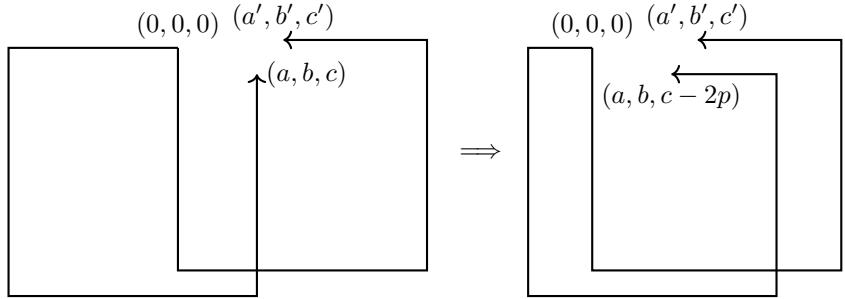
\begin{center}
\begin{tikzpicture}[scale=0.35]
\draw[thick,->] (0, 0) -- (-6.433981132056603, 0) -- (-6.433981132056603, -9.433981132056603) -- (3, -9.433981132056603) -- (3, -1);
\draw[thick,->] (0, 0) -- (0, -8.460443964212251) -- (9.460443964212251, -8.460443964212251) -- (9.460443964212251, 0.3) -- (4, 0.3);
\node[above] () at (0,0) {$(0,0,0)$};
\node[right] () at (3,-1) {$(a,b,c)$};
\node[above] () at (4,0.3) {$(a',b',c')$};
\node() at (11.5,-4) {$\implies$};
\end{tikzpicture}
\begin{tikzpicture}[scale=0.35]
\draw[thick,->] (0, 0) -- (-2.433981132056603, 0) -- (-2.433981132056603, -9.433981132056603) -- (7, -9.433981132056603) -- (7, -1) -- (3, -1);
\draw[thick,->] (0, 0) -- (0, -8.460443964212251) -- (9.460443964212251, -8.460443964212251) -- (9.460443964212251, 0.3) -- (4, 0.3);
\node[above] () at (0,0) {$(0,0,0)$};
\node[below] () at (3,-1) {$(a,b,c-2p)$};
\node[above] () at (4,0.3) {$(a',b',c')$};
\end{tikzpicture}
\end{center}
The vertical difference stays $O(1)$, while we can add make the geodesic end with an arbitrarily long suffix of $\wst$-moves. Since adding the segment of length $k$ only increases the height difference linearly in $k$, the argument from the second paragraph applies.

Of course, a small argument is needed to show that $(a,b,c-2p)$ is in the $n$-ball, and that we can move down in height within in the $n$-ball. To step down by $2$ in $12$ steps, walk down along the canonical geodesic of $g$ for $4$ steps, step down in height by moving around a cycle, and then walk back along the reverse of the same geodesic path of length $4$. The formula for the word metric shows that this path stays fully inside the ball since the height stays large at all times due to $c \geq t(a + b)$ (note that the \emph{columns}, i.e.\ fibers of the natural projection, are connected by the formula for the word metric).

Finally, if there is a bump, without changing the area, endpoints or perimeter (thus keeping geodesicity) we can repeatedly erase the bump from the west and introduce a bump of equal length on the east, until the bump on the west is longer than the eastmost side. By moving down in height, we can drop the length of the bump and keep moving it on the east side.
 
We have now dealt with the cases where $|a|+|b|$ is large enough and $c \geq t(|a| + |b|), |c'| \geq t(|a'| + |b'|)$ for some (unspecified, large) $t$. For any fixed $t$, we can fix with the cases where $c \leq t(|a| + |b|)$ or $c' \leq t(|a'| + |b'|)$ for large $|a| + |b|$ easily: The pseudogeodesics for $(a,b,c)$ and $(a',b',c')$ 
both stay at a bounded distance from the line from $(0,0)$ to $(a,b)$, so we can apply the argument from the second paragraph.

We have now dealt with all cases where $|a|+|b|$ is large enough. We still need to bound distortion near the origin. It is now useful to think of also the bound for $|a|$ and $|b|$ as a constant, i.e.\ $|a|, |b| = O(1)$. Consider such $g = (a,b,c), g' = (a',b',c')$ with $d(g, g') \leq m$, where $|a|, |b|, |a'|, |b'| = O(1)$. Again by symmetry we may assume $a, b, c \geq 0$. Necessarily $|c|$ (thus $c$ by assumption) is very large if $n_0$ is, as otherwise we are nowhere near the $n$-sphere and any geodesic from $g$ to $h$ is contained inside $B_{S,n}$. Note that necessarily $c'$ is also large.

Now, walk back along the canonical geodesics from $g$ or $g'$ (for some bounded number of steps) so that both are on the same $\ell^1$-sphere in the natural projection ($c, c'$ are large, so the geodesics cannot stay close to the origin for a long time). Note that since $c$ and $c'$ are large, $g$ is connected to $g (0,0,-2)$ and $g'$ is connected to $g'(0,0,-2)$ in at most $12$ steps inside the $n$-ball.

As in the previous section, consider the function $\eta(x,y) = \max \{z \;|\; (x, y, z) \in B_{S, n}\}$ and recall that $\eta(x,y)$ depends only $|x| + |y|$ up to an additive constant $C$. As in the previous section, we see that as $|x| + |y|$ increases, $\eta(x, y)$ grows linearly. Due to parity issues, it may not grow on every step. Nevertheless, it is clear that the height cannot drop by more than by a constant amount in a single step. Thus, for large $n$, for two elements $g, g'$ whose projections are on the same $\ell^1$-sphere, we can first drop down in height by a sufficient constant amount (as we did above to move from $(a,b,c)$ to $(a,b,c-2p)$). Then any shortest path, which does not enter the $\ell^1$-ball encircled by the $\ell^1$-sphere containing $g$ and $g'$, is contained in the $n$-ball.
\end{proof}

It seems likely that there are shorter proofs than the one given here, but it seems unlikely that there is a completely trivial proof, since the geometric phenomena guiding distortion only appear on the large scale. For example, the following paths are the geodesics from $(0,4,24)$ to $(0,6,24) = (0,4,24) \nth\nth$ inside the $10$-ball -- the maximal height in the column at $(0,5)$ is $\eta(0,5) = 20$ so $(0,4,24) \nth = (0,5,24)$ is not in this ball, and we have to take one of four detours.

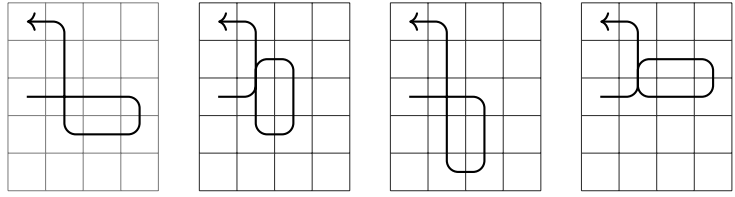
\begin{center}
\begin{tikzpicture}[baseline = 0,scale=0.5]
\draw[gray,shift={(0.5,0.5)}] (-1,-3) grid (3,2);
\draw[thick, rounded corners=1ex,->] (0,0) -- ++(1,0) -- ++(1,0) -- ++(1,0) -- ++(0,-1) -- ++(-1,0) -- ++(-1,0) -- ++(0,1) -- ++(0,1) -- ++(0,1) -- ++(-1,0);
\end{tikzpicture}
\;\;\;
\begin{tikzpicture}[baseline = 0,scale=0.5]
\draw[white!20!black,shift={(0.5,0.5)}] (-1,-3) grid (3,2);
\draw[thick, rounded corners=1ex,->] (0,0) -- ++(1,0) -- ++(0,1) -- ++(1,0) -- ++(0,-1) -- ++(0,-1) -- ++(-1,0) -- ++(0,1) -- ++(0,1) -- ++(0,1) -- ++(-1,0);
\end{tikzpicture}
\;\;\;
\begin{tikzpicture}[baseline = 0,scale=0.5]
\draw[white!20!black,shift={(0.5,0.5)}] (-1,-3) grid (3,2);
\draw[thick, rounded corners=1ex,->] (0,0) -- ++(1,0) -- ++(1,0) -- ++(0,-1) -- ++(0,-1) -- ++(-1,0) -- ++(0,1) -- ++(0,1) -- ++(0,1) -- ++(0,1) -- ++(-1,0);
\end{tikzpicture}
\;\;\;
\begin{tikzpicture}[baseline = 0,scale=0.5]
\draw[white!20!black,shift={(0.5,0.5)}] (-1,-3) grid (3,2);
\draw[thick, rounded corners=1ex,->] (0,0) -- ++(1,0) -- ++(0,1) -- ++(1,0) -- ++(1,0) -- ++(0,-1) -- ++(-1,0) -- ++(-1,0) -- ++(0,1) -- ++(0,1) -- ++(-1,0);
\end{tikzpicture}
\end{center}

The proof given by Shapiro in \cite{Sh89} for the Heisenberg group being almost convex does not look essentially simpler than our proof above, but we admit that we have not studied it in detail.

Based on calculations with small balls, we conjecture that the initial values of the distortion function of large enough balls (in the sense of Lemma~\ref{lem:DisconnectedHoroball}) are
\[ (\ell(1), \ell(2), \ell(3), \ell(4), \ell(5)) = (1, 10, 11, 16, 17). \]
The vague hunch of the author is that $\ell$ grows linearly.

Note that \cite{Ca87} shows that as soon as $\ell(2)$ is finite, all values of $\ell$ are finite, and one can obtain an upper bound directly from the argument, but it is not clear how to obtain a subexponential bound.

\subsection{Calculations for $f(a,b,c) \leq (a,b,c)$}
\label{sec:Calculations}

We complete the calculations needed for the proof of Lemma~\ref{lem:WordMetric} and show that $f(a,b,c) \leq (a,b,c)$ when $a,b,c \geq 0$.

Consider first $(a,b,c) \wst = (a-1,b,c+b) = (a',b',c')$.  Suppose that $f(a',b',c') \leq f(a,b,c) - 2$. Suppose first that additionally $a' \geq 0$. Then
\begin{align*}
f(a',b',c') &= \min \{ 2(A + B) - (a' + b') \;|\; 2AB - a'b' \geq c', A \geq a', B \geq b' \} \\
&= \min \{ 2(A + B) - (a-1 + b) \;|\; 2AB - (a-1)b \geq c+b, A \geq a-1, B \geq b \} \\
&= \min \{ 2(A + B) - (a + b) + 1 \;|\; 2AB - ab \geq c, A \geq a-1, B \geq b \} \\
&\leq f(a,b,c) - 2 \\
&= \min \{ 2(A + B) - (a + b) - 2 \;|\; 2AB - ab \geq c, A \geq a, B \geq b \} \\
&\leq \min \{ 2(a + B) - (a + b) - 2 \;|\; 2aB - ab \geq c, B \geq b \} 
\end{align*}
where we restrict to $A = a$ in the last inequality. The minimum for $f(a',b'c')$ must be reached at $A = a-1$ (or we could not possibly have even $f(a',b',c') \leq f(a,b,c)$), so we get
\begin{align*}
f(a',b',c') &= \min \{ 2(A + B) - (a + b) + 1 \;|\; 2AB - ab \geq c, A \geq a-1, B \geq b \} \\
&= \min \{ 2(a - 1 + B) - (a + b) + 1 \;|\; 2(a-1)B - ab \geq c, B \geq b \} \\
&= \min \{ 2(a + B) - (a + b) - 1 \;|\; 2aB - 2B - ab \geq c, B \geq b \} \\
&\leq \min \{ 2(a + B) - (a + b) - 2 \;|\; 2aB - ab \geq c, B \geq b \} 
\end{align*}
where the last inequality comes from the chain of inequalities above. This is absurd, since $b \geq 0$ implies that $2aB - 2B - ab \geq c$ is a more restrictive condition than $2aB - ab \geq c$.

Suppose next that $a' < 0$, so in fact $a = 0$ and $|a'| = |-1| = 1$. Then $f(a',b',c') \leq f(a,b,c) - 2$ implies
\begin{align*}
f(a',b',c') &= \min \{ 2(A + B) - b - 1 \;|\; 2AB - b \geq c+b, A \geq 1, B \geq b \} \\
&\leq \min \{ 2(A + B) - b - 2 \;|\; 2AB \geq c, A \geq 0, B \geq b \}
\end{align*}
which is preposterous, because $2AB - b \geq c+b \wedge A  \geq 1$ is a more restrictive condition than $2AB \geq c \wedge A \geq 0$.

Consider then $(a,b,c) \est = (a+1,b,c-b) = (a',b',c')$. Suppose first that $a \geq 1$ so that $c' \geq 0$. If $f(a',b',c') \leq f(a,b,c) - 2$ then
\begin{align*}
f(a',b',c') &= \min \{ 2(A + B) - (a' + b') \;|\; 2AB - a'b' \geq c', A \geq a', B \geq b' \} \\
&= \min \{ 2(A + B) - (a + 1 + b) \;|\; 2AB - (a+1)b \geq c-b, A \geq a+1, B \geq b \} \\
&= \min \{ 2(A + B) - (a + b) - 1 \;|\; 2AB - ab \geq c, A \geq a+1, B \geq b \} \\
&\leq f(a,b,c) - 2 \\
&= \min \{ 2(A + B) - (a + b) - 2 \;|\; 2AB - ab \geq c, A \geq a, B \geq b \}
\end{align*}
which is insane, since the condition $A \geq a+1$ is more restrictive than $A \geq a$.

Suppose then that $a = 0$ and $c - b < 0$. Then
\begin{align*}
f(a',b',c') &= \min \{ 2(A + B) - (a' + b') \;|\; 2AB - a'b' \geq |c'|, A \geq a', B \geq b' \} \\
&= \min \{ 2(A + B) - (a + 1 + b) \;|\; 2AB - (a+1)b \geq b-c, A \geq a+1, B \geq b \} \\
&= \min \{ 2(A + B) - (a + b) - 1 \;|\; 2AB - ab \geq 2b - c, A \geq a+1, B \geq b \} \\
&\leq f(a,b,c) - 2 \\
&= \min \{ 2(A + B) - (a + b) - 2 \;|\; 2AB - ab \geq c, A \geq a, B \geq b \}
\end{align*}
But this is ridiculous, since $c - b < 0 \implies 2b - c > c$, so the former set of conditions is at least as restrictive as the latter.

Consider then $(a,b,c) \nth = (a,b+1,c+a) = (a',b',c')$. Then
\begin{align*}
f(a',b',c') &= \min \{ 2(A + B) - (a' + b') \;|\; 2AB - a'b' \geq c', A \geq a', B \geq b' \} \\
&= \min \{ 2(A + B) - (a + b) - 1 \;|\; 2AB - ab \geq c + 2a, A \geq a, B \geq b+1 \} \\
&\leq f(a,b,c) - 2 \\
&= \min \{ 2(A + B) - (a + b) - 2 \;|\; 2AB - ab \geq c, A \geq a, B \geq b \}
\end{align*}
is pure lunacy.

Consider then $(a,b,c) \sth = (a,b-1,c-a) = (a',b',c')$. Suppose first that additionally $b' \geq 0$ and $c - a \geq 0$. Then
\begin{align*}
f(a',b',c') &= \min \{ 2(A + B) - (a' + b') \;|\; 2AB - a'b' \geq c', A \geq a', B \geq b' \} \\
&= \min \{ 2(A + B) - (a + b) + 1 \;|\; 2AB + 2a - ab \geq c, A \geq a, B \geq b-1 \} \\
&\leq f(a,b,c) - 2 \\
&= \min \{ 2(A + B) - (a + b) - 2 \;|\; 2AB - ab \geq c, A \geq a, B \geq b \} \\
&= \min \{ 2(A + B + 1) - (a + b) - 2 \;|\; 2A(B+1) - ab \geq c, A \geq a, B+1 \geq b \} \\
&= \min \{ 2(A + B) - (a + b) \;|\; 2AB + 2A - ab \geq c, A \geq a, B \geq b-1 \}.
\end{align*}
Under $A \geq a, B \geq b-1$, the condition $2AB + 2a - ab \geq c$ is stricter than $2AB + 2A - ab \geq c$, so this is absolute madness.

Suppose next that $b' < 0$, so in fact $b = 0$ and $b' = -1$, but $c - a \geq 0$. Then
\begin{align*}
f(a',b',c') &= \min \{ 2(A + B) - (a + 1) \;|\; 2AB - a \geq c - a, A \geq a, B \geq 1 \} \\
&= \min \{ 2(A + B) - a - 1 \;|\; 2AB \geq c, A \geq a, B \geq 1 \} \\
&\leq f(a,b,c) - 2 \\
&= \min \{ 2(A + B) - a - 2 \;|\; 2AB \geq c, A \geq a, B \geq 0 \}
\end{align*}
is just nuts.

Suppose then that $b' \geq 0$ but $c - a < 0$. Then
\begin{align*}
f(a',b',c') &= \min \{ 2(A + B) - (a' + b') \;|\; 2AB - a'b' \geq c', A \geq a', B \geq b' \} \\
&= \min \{ 2(A + B) - (a + b) + 1 \;|\; 2AB - a(b - 1) \geq a - c, A \geq a, B \geq b-1 \} \\
&= \min \{ 2(A + B) - (a + b) + 1 \;|\; 2AB - ab \geq 2a - c, A \geq a, B \geq b-1 \} \\
&\leq f(a,b,c) - 2 \\
&= \min \{ 2(A + B) - (a + b) - 2 \;|\; 2AB - ab \geq c, A \geq a, B \geq b \}.
\end{align*}
Since $c - a < 0 \implies 2a - c > c$, $f(a',b',c')$ must be minimized at $B = b-1$, so
\begin{align*}
f(a', b', c') &= \min \{ 2(A + b - 1) - (a + b) + 1 \;|\; 2A(b-1) - ab \geq 2a - c, A \geq a \} \\
&= \min \{ 2(A + b) - (a + b) - 1 \;|\; 2Ab - 2A - ab \geq 2a - c, A \geq a \} \\
&\leq \min \{ 2(A + b) - (a + b) - 2 \;|\; 2Ab - ab \geq c, A \geq a \}.
\end{align*}
where the last inequality is $f(a,b,c)-2$ specialized to $B = b$. But that's crazy.

Suppose finally that $b' < 0$ but $c - a < 0$.
\begin{align*}
f(a',b',c') &= \min \{ 2(A + B) - (a' + |b'|) \;|\; 2AB - a'b' \geq |c'|, A \geq a', B \geq b' \} \\
&= \min \{ 2(A + B) - a - 1 \;|\; 2AB \geq 2a - c, A \geq a, B \geq 1 \} \\
&\leq f(a,b,c) - 2 \\
&= \min \{ 2(A + B) - a - 2 \;|\; 2AB \geq c, A \geq a, B \geq 0 \}.
\end{align*}
Sulaa hulluutta.

\bibliographystyle{plain}
\bibliography{../../../bib/bib}{}

\end{document}